\documentclass[12pt]{article}
\usepackage{mathrsfs}
\usepackage{amssymb}
\usepackage{}
\usepackage{amsfonts}
\usepackage{amsfonts,amsmath,amsthm, amssymb}
\usepackage{array}
\usepackage{latexsym, euscript, epic, eepic}
\usepackage{tikz}
\usepackage[noblocks]{authblk}
\usetikzlibrary{matrix}
\newtheorem{lemma}{Lemma}
\newtheorem{theorem}{Theorem}

\newtheorem{remark}{Remark}

\begin{document}
 
\title{Ahlfors-regular curves and Carleson measures }

\author{Huaying Wei \thanks{Department of Mathematics and Statistics, Jiangsu Normal University, Xuzhou 221116, PR China. Email:  6020140058@jsnu.edu.cn.  } ,   Michel Zinsmeister \thanks{IDP, Universit\'e d' Orl\'eans,45067 Orl\'eans  Cedex 2, France. Email: zins@univ-orleans.fr}}
\date{}
\maketitle

\begin{center}
\begin{minipage}{120mm}
{\small{\bf Abstract}.
We study the relation between the boundary of a simply connected domain $\Omega$ being Ahlfors-regular and  the invariance of Carleson measures under the push-forward operator induced by a conformal  mapping from the unit disk $\Delta$ onto the domain $\Omega$. As an application, we characterize the chord-arc curve with small norm and the asymptotically smooth curve in terms of the complex dilatation of some quasiconformal reflection with respect to the curve. 
}

\end{minipage}
\end{center}


 \vskip1cm

\section{Introduction}
Let  $\Delta = \{z:\,  |z| < 1\}$ denote the unit disk in the extended complex plane $\hat{\mathbb{C}}$. $\Delta^{*} = \hat{\mathbb{C}}-\bar{\Delta}$ is the exterior of $\Delta$, and $\mathbb{S} = \partial\Delta = \partial\Delta^{*}$ is the unit circle.  Also, as usual, 
$D(z, r)$ denotes the open disk with center $z$ and radius $r$, and the notation $\Lambda^{1}$ denotes the one-dimensional Hausdorff measure. 

A positive measure $\mu$ defined on a simply connected domain $\Omega$ is called a Carleson measure (see \cite{Ga}) if
\begin{equation}\label{carlesonnorm}
\|\mu\|_{*} = \sup \Big\{\frac{\mu(\Omega \cap D(z, r))}{r}: z \in \partial \Omega, 0 < r < diameter(\partial\Omega)\Big\} < \infty.
\end{equation}
A Carleson measure $\mu$ is called a vanishing Carleson measure if $\lim_{r\rightarrow 0}\mu(\Omega \cap D(z, r))/r = 0$ uniformly for $z \in \partial \Omega$. We denote by $CM(\Omega)$ and $CM_0(\Omega)$ the set of all Carleson measures and vanishing Carleson measures on $\Omega$, respectively. It is easy to see that $CM(\Omega)$ is a Banach space with the Carleson norm $\|\cdot\|_{*}$.

Let $\varphi$ be a conformal mapping from the unit disk $\Delta$ onto a simply connected domain $\Omega$. For any $\mu \in CM(\Omega)$, the pull-back of $\mu$ is the measure defined on $\Delta$ by
$$ \varphi^{*}d\mu=|\varphi^{'}|^{-1}d(\mu\circ \varphi).$$
It should be pointed out that if $\mu \in CM(\Omega)$ is absolutely continuous (with respect to Lebesgue measure), that is,  if  there exists a function $\lambda \in L^{1}$ such that 
$$d\mu(z) = \lambda(z)dxdy,$$
then, writing $d\nu=\varphi^*d\mu$, $$d\nu(\zeta) = \lambda\circ\varphi(\zeta) |\varphi^{'}(\zeta)|d\xi d\eta.$$

For a Carleson measure $\nu$ on $\Delta$ we define similarly the push-forward of $\nu$ as being the measure on $\Omega$ defined by 
$$(\varphi^{-1})^{*}d\nu = |(\varphi^{-1})^{'}|^{-1}d(\nu\circ \varphi^{-1}).$$

If $\Omega=\Delta$, these two operators are isomorphisms of $CM(\Delta)$  (one being the reciprocal of the other): this is another way of stating  the conformally invariant character of Carleson measures on $\Delta$ as in \cite[p.231]{Ga}. 

In 1989, Zinsmeister \cite{Zi89} proved the following
\begin{theorem}\label{Z1}
Let $\varphi$ be a conformal mapping from the unit disk $\Delta$ onto a simply connected domain $\Omega$. Then $\log \varphi^{'} \in BMOA(\Delta)$ if and only if the pull-back operator $\varphi^{*}$ is bounded from $CM(\Omega)$ to $CM(\Delta)$.
\end{theorem}

Let $\varphi$ be a univalent function in $\Delta$. We say that $\Omega = \varphi(\Delta)$ is a BMOA domain if $\log \varphi^{'} \in BMOA(\Delta)$. For the BMOA domain, Bishop and Jones \cite{BJ} gave a geometric characterization. It says that $\log \varphi^{'} \in BMOA(\Delta)$ if and only if the domain $\Omega = \varphi(\Delta)$ satisfies the following Bishop-Jones (BJ) condition:

 For any $z\in\Omega$ there exists a chord-arc domain  $\Omega_z\subset\Omega$ containing $z$ of ``norm"
 $\leqslant k(\Omega)$, whose diameter is uniformly comparable to $dist(z,\partial\Omega)$, and such that  $\Lambda^1(\partial\Omega\cap \partial\Omega_z)\geq c(\Omega)dist(z,\partial\Omega).$ 
 Here $k(\Omega)>1$ and $c(\Omega)>0$ depend only on $\Omega$, and $\Omega_z$  being a chord-arc domain means that its boundary is a chord-arc curve. 

A Jordan curve is a BJ curve if it is the boundary of a simply connected domain which satisfies the above BJ condition. 
A BJ curve which is a quasicircle is called a BJ quasicircle. A typical example is a variant of the snowflake where at each iteration step, one of sides of the triangle, for instance the left one, is left unchanged. 

Very recently, in \cite{WZ} we proved the analogous statement of Theorem \ref{Z1} for vanishing Carleson measures.
\begin{theorem}\label{Delta}
Let $\varphi$ be a conformal mapping from the unit disk $\Delta$ onto a simply connected domain $\Omega$.  Assume that  $\partial \Omega$ is  a BJ quasicircle. Then the pull-back operator
\begin{equation*}
\varphi^{*}: \; CM_0(\Omega) \to CM_0(\Delta)
\end{equation*}
is well-defined.
\end{theorem}

Recall that a locally rectifiable curve $\Gamma \subset \mathbb{C}$ is Ahlfors-regular if  there exists a constant $C > 0$ such that its arclength parametrization 
$z: I \to \mathbb{C}$  $(I; \;interval\; of\; \mathbb{R})$ satisfies 
\begin{equation}\label{def}
 \forall z \in \mathbb{C},\, \forall r > 0, \,\Lambda^{1}(\{s \in I;\; |z(s) - z| \leqslant r \}) \leqslant C r.
\end{equation}
The smallest such $C$ is called the Ahlfors-regular constant.  It has been shown by David  (see  \cite[p.162]{Po92})  that  $\Gamma$ is Ahlfors-regular if and only if 
\begin{equation}\label{def2}
\Lambda^{1}(\tau(\Gamma)) \leqslant C_1\, diam \tau(\Gamma)
 \end{equation}
for all $\tau \in $ M\"ob.

This concept was introduced by David \cite{Dav} in association with the problem of Cauchy integral. More precisely, This author demonstrated that Ahlfors-regular curves are the most general rectifiable curves for which for all $f \in L^2(\mathbb{R})$, 
\begin{equation*}
\lim_{\varepsilon \to 0}\int_{|z(y) - z(x)| > \varepsilon} \frac{f(y)dy}{z(x) - z(y)} = Tf(x)
\end{equation*}
exists almost everywhere and defines a bounded operator on $L^2(\mathbb{R})$.

Another result in \cite{Zi89} states that  $\partial \Omega$ being Ahlfors-regular is a sufficient condition to assure the boundedness of the push-forward operator $(\varphi^{-1})^{*}$ from $CM(\Delta)$ to $CM(\Omega)$. We will show that the converse is also true. Thus, the invariance of Carleson measures under push-forward operators characterizes Ahlfors-regular curves.  More precisely, 
\begin{theorem}\label{Z2}
Let $\varphi$ be a conformal mapping from the unit disk $\Delta$ onto a simply connected domain $\Omega$.  
Then  $\partial \Omega$ is Ahlfors-regular if and only if the push-forward operator $(\varphi^{-1})^{*}$ is well-defined  from $CM(\Delta)$ to $CM(\Omega)$.
\end{theorem}

We will show as part of the proof of Theorem \ref{Z2} that the push-forward operator $(\varphi^{-1})^{*}$ being well-defined assures that the image of an Ahlfors-regular curve in $\Delta$ under a conformal mapping is also an Ahlfors-regular curve.  The corresponding result for the pull-back operator, the boundedness of the pull-back operator implies that the inverse mapping of a conformal mapping from the unit disk $\Delta$ onto a simply connected domain $\Omega$ preserves being Ahlfors-regular, is already known.

In \cite{WZ}, we showed the analogous statement of Theorem \ref{Z2} for vanishing Carleson measures is still valid.
\begin{theorem}\label{Omega}
Let $\varphi$ be a conformal mapping from the unit disk $\Delta$ onto a simply connected domain $\Omega$.  Assume that  $\partial \Omega$ is an Ahlfors-regular curve. Then the push-forward operator 
$$(\varphi^{-1})^{*}: \; CM_0(\Delta) \to CM_0(\Omega)$$ is well-defined.
\end{theorem}

The result in \cite{Zi84} states that an Ahlfors-regular curve must be a BJ curve, and the converse does not hold. A counterexample can be constructed as follows: Let 
$\Omega = \{(x, y) \in \mathbb{R}^2: y > sin(x^2)\}$. Then $\partial\Omega$ is not an Ahlfors-regular curve, but $\log \varphi^{'} \in BMOA(\mathbb{R}^2_{+})$, where $\varphi$ is a conformal mapping from $\mathbb{R}^2_{+}$ onto $\Omega$.

On the other hand, the BJ quasicircle might not even be rectifiable, though in a sense, it is rectifiable most of the time on all scales. According to the observation that a curve is both an Ahlfors-regular curve and a quasicircle, so called Ahlfors-regular quasicircle,  if and only if 
it is a chord-arc curve, it is clear that  an Ahlfors-regular quasicircle must be rectifiable. 

The paper is structured as follows: Section 2 is denoted to the proof of Theorem \ref{Z2}. As an application of Carleson measure theories stated in Section 1,  in Section 3, we will characterize chord-arc curves and asymptotically smooth curves in terms of quasiconformal  reflections.

\section{The boundedness of the push-forward operator}

Before proceeding to the proof of Theorem \ref{Z2}, let us recall its statement. 
\addtocounter{theorem}{-2}
\begin{theorem}
Let $\varphi$ be a conformal mapping from the unit disk $\Delta$ onto a simply connected domain $\Omega$.  
Then the following two statements are equivalent:
\begin{description}
\item[(1)]  $\partial \Omega$ is Ahlfors-regular; 
\item[(2)] The push-forward operator $(\varphi^{-1})^{*}$ is  well-defined  from $CM(\Delta)$ to $CM(\Omega)$.
\end{description}
\end{theorem}

The proof of $(1) \Rightarrow (2)$ in Theorem 3 has been shown in \cite{Zi89}. 
 The strategy for the proof of $(2) \Rightarrow (1)$ is to use the fact that if the curve $\gamma \subset \Delta$ is Ahlfors-regular, then $\Lambda^{1}|_{\gamma}$ is a Carleson measure on the unit disk $\Delta$. 

We begin by stating  Koebe distortion theorem, which will be needed later. See \cite[Theorem 1.3]{Po92} for example.
\begin{lemma}
A conformal homeomorphism $f$ of $\Delta$ into $\mathbb{C}$ satisfies
\begin{equation*}
\begin{split}
& |f^{'}(0)|\frac{|z|}{(1 + |z|)^2} \leqslant |f(z) - f(0)| \leqslant |f^{'}(0)|\frac{|z|}{(1 - |z|)^2}; \\
&|f^{'}(0)|\frac{1 - |z|}{(1 + |z|)^3} \leqslant |f^{'}(z)| \leqslant |f^{'}(0)|\frac{1 + |z|}{(1 - |z|)^3}\\
\end{split}
\end{equation*}
for every $z \in \Delta$. The first inequality of the first line  shows in particular that the image $f(\Delta)$ contains a disk of center at $f(0)$ and radius $|f^{'}(0)|/4$.
\end{lemma}

\begin{lemma}
Let $\varphi$ be a conformal mapping from the unit disk $\Delta$ onto a simply connected domain $\Omega$. Assume that the push-forward operator $(\varphi^{-1})^{*}$ is  well-defined  from $CM(\Delta)$ to $CM(\Omega)$. 
Then $(\varphi^{-1})^{*}$ is a bounded linear operator.
\end{lemma}
\begin{proof}
Let
$$T=(\varphi^{-1})^{*}: \; CM(\Delta) \to CM(\Omega)$$
which sends by hypothesis $\nu$ to $\mu$ defined by $d\mu = |(\varphi^{-1})^{'}|^{-1} d\nu\circ\varphi^{-1}$.  We define the graph of $T$ to be the set 
$$\Gamma(T) = \{(\nu, \mu) \in CM(\Delta) \times CM(\Omega); \, T(d\nu) = d\mu\}.$$
Let $\{(\nu_n, \mu_n)\} \subset \Gamma(T)$ such that $\nu_n \to \nu,\,\mu_n\to\mu$ as $n \to \infty$. If we can show that $d\mu = T(d\nu)$ then, by the closed graph theorem the conclusion of the lemma holds. 

The assumptions imply that $(\mu_n)$ converges weakly to $\mu$ in the sense that
$$\int fd\mu_n\to \int fd\mu$$
for every $f$ continuous with compact support in $\Omega$. On the other hand,
$$ \int f d\mu_n=\int f\circ\varphi\vert \varphi^{'}\vert d\nu_n\to \int f\circ\varphi\vert \varphi^{'} \vert d\nu=\int f T(d\nu),$$
so that $d\mu=T(d\nu)$ and the lemma follows. 
\end{proof}
\begin{lemma}
Let $\varphi$ be a conformal mapping from the unit disk $\Delta$ onto a simply connected domain $\Omega$. Assume that the push-forward operator $(\varphi^{-1})^{*}$ is  well-defined  from $CM(\Delta)$ to $CM(\Omega)$. If $\gamma \subset \Delta$ is Ahlfors-regular, then $\Gamma = \varphi(\gamma) \subset \Omega$ is Ahlfors-regular.
\end{lemma}

\begin{proof}
Suppose that the curve $\gamma \subset \Delta$ is Ahlfors-regular.  We define $\nu = \Lambda^{1}|_{\gamma}$ and denote by $\mu$ the push-forward of the measure $\nu$. 
By (\ref{def}), there exists some constant $C > 0$ such that for any $z \in \mathbb{C}$ and $r > 0$, we have 
\begin{equation}\label{Ar}
\nu (D(z, r)\cap \gamma) = \Lambda^{1}(D(z, r)\cap \gamma) \leqslant C r.
\end{equation} 
In particular, (\ref{Ar}) holds for any $z \in \mathbb{S}$ and $0 < r < 2$, 
which implies $\nu \in CM(\Delta)$. For any domain $A \subset \Omega$, 
\begin{equation*}
\begin{split}
\mu(A) & =  \int_{A}d\mu = \int_{A}\frac{1}{|(\varphi^{-1})^{'}|}d \nu \circ \varphi^{-1} \\
&= \int_{\varphi^{-1}(A)}\frac{1}{|(\varphi^{-1})^{'}|(\varphi(z))}d\nu = \int_{\varphi^{-1}(A)}|\varphi^{'}(z)|d\nu\\
& = \int_{\varphi^{-1}(A) \cap \gamma}|\varphi^{'}(z)| |dz| = \Lambda^{1}(\varphi(\gamma) \cap A) = \Lambda^{1}(\Gamma \cap A).\\
\end{split}
\end{equation*}
Thus, $\mu = \Lambda^1|_{\Gamma} \in CM(\Omega)$.

Let $w_0 \in \Gamma$:  Suppose $r \geqslant \frac{1}{2}dist(w_0, \partial\Omega)$. Then $r + dist(w_0, \partial\Omega) \leqslant 3r.$ Set 
$dist(w_0, \partial\Omega) = |w_0 - \tilde{w_0}|, \; \tilde{w_0} \in \partial\Omega$. It is easy to see that 
\begin{equation}
\begin{split}
\Lambda^{1}( D(w_0, r) \cap \Gamma) &= \mu(D(w_0, r) \cap \Gamma)\\  
&\leqslant \mu(D(\tilde{w_0}, 3r) \cap \Gamma) \\
&\leqslant 3\|\mu\|_{*} r \leqslant 3C\|T\|r.\\
\end{split}
\end{equation}
The last inequality follows from (\ref{Ar}) and the boundedness of the push-forward operator  $T = (\varphi^{-1})^{*}$, where $\|T\|$ denotes the operator norm.

Suppose $r < \frac{1}{2}dist(w_0, \partial\Omega)$. 
Set $d = dist(w_0, \partial\Omega)$. We have $D(w_0, r) \subset D(w_0, d) \subset \Omega.$ Set $\phi = \varphi^{-1}$. 
Then $\phi(dw + w_0)$ is a conformal homeomorphism of $\Delta$ into $\Delta$.  By Koebe distortion theorem, for any $w \in D(w_0, r)$, we have 
\begin{equation}\label{Koebe1}
\frac{4}{27}|\phi^{'}(w_0)| \leqslant |\phi^{'}(w)| \leqslant 12|\phi^{'}(w_0)|
\end{equation}
and 
\begin{equation}\label{Koebe2}
|\phi(w_0 + re^{i\theta}) - \phi(w_0)| \leqslant d|\phi^{'}(w_0)|4\frac{r}{d} = 4r|\phi^{'}(w_0)|.
\end{equation}
Then, (\ref{Koebe2}) implies that
\begin{equation}\label{4r}
\phi(D(w_0, r)) \subset D(\phi(w_0), 4r|\phi^{'}(w_0)|).
\end{equation}
We conclude by (\ref{Ar}) and (\ref{4r}) that, 
\begin{equation}\label{1}
\Lambda^{1}(\phi(D(w_0, r) \cap \Gamma)) \leqslant 4C r |\phi^{'}(w_0)|.
\end{equation}
On the other hand, it follows from (\ref{Koebe1}) that 
\begin{equation}\label{2}
\begin{split}
\Lambda^{1}(\phi(D(w_0, r) \cap \Gamma)) & = \int_{D(w_0, r) \cap \Gamma}|\phi^{'}(w)| |dw|\\ & \geqslant \frac{4}{27}|\phi^{'}(w_0)| \Lambda^{1}(D(w_0, r) \cap \Gamma).\\
\end{split}
\end{equation}
Then, by (\ref{1}) and (\ref{2}) we have 
$\Lambda^{1}(D(w_0, r) \cap \Gamma) \leqslant 27C r.$

Let $w_0 \in \mathbb{C}\setminus \Gamma$: if $r \leqslant dist(w_0, \Gamma)$, there is nothing to prove. 

If $r > dist(w_0, \Gamma)$, let $dist(w_0, \Gamma) = |w_0 - \tilde{w}|, \, \tilde{w} \in \Gamma.$ Then,
$$\Lambda^{1}(D(w_0, r) \cap \Gamma) \leqslant \Lambda^1 (D(\tilde{w}, 2r) \cap \Gamma) \leqslant 2max\{3C \|T\|, 27C\}r.$$
Therefore, we conclude that  
$\Lambda^1(D(w, r) \cap \Gamma) \leqslant 2max\{3C \|T\|, 27C\}r$
for any $w \in \mathbb{C}$ and $r \in (0, \infty)$. This completes the proof of Lemma 3. 
\end{proof}

\begin{remark}
For the ease of calculation, the Carleson measure on $\Delta$ can be defined in the following equivalent way. A positive measure $\mu$ on $\Delta$ is a Carleson measure if there is a constant $C$ such that 
\begin{equation}
\mu(S) \leqslant Ch
\end{equation}
for every sector 
\begin{equation*}
S = \{re^{i\theta}: 1 - h \leqslant r < 1, |\theta - \theta_0| \leqslant h\}.
\end{equation*}
The smallest such $C$ is called the Carleson norm of $\mu$, $\|\mu\|_{C}$. For the sake of  simplicity, the norm $\|\cdot\|_{*}$ in (\ref{carlesonnorm}) and the norm $\|\cdot\|_{C}$ will be identified in the following arguments. 
\end{remark}

We are ready now to the proof of $(2) \Rightarrow (1)$ in Theorem  \ref{Z2}.
\begin{proof}[proof of $(2) \Rightarrow (1)$ in  Theorem \ref{Z2}] We first suppose $\partial \Omega$ is bounded.  
Set $\gamma_n = \{z: \; |z| = r_n = 1 - 2^{-n}\}, \; n = 1, 2, 3, \cdots$ and $\varphi(\gamma_n) = \Gamma_n$. We define $\nu_n = \Lambda^1|_{\gamma_n}$. By direct  computation we see $\|\nu_n\|_{*} \leqslant 1$ for any $n$. Then $d\mu_n = (\varphi^{-1})^{*}d\nu_n = d\Lambda^1|_{\Gamma_n}$ follows from the proof of Lemma 3. We conclude  by Lemma 2 that $\|\mu_n\|_{*} \leqslant \|T\|$ for any $n$, where $\|T\|$ denotes  the operator norm of $T = (\varphi^{-1})^{*}$ as before.

For any $w \in \partial\Omega$,
\begin{equation*} 
\begin{split}
 \Lambda^{1}(\Gamma_n) & =\mu_n(D(w, diam(\partial\Omega)) \cap \Gamma_n)\\
&= \int_{|z| = r_n}|\varphi^{'}(z)||dz| = \int_0^{2\pi}|\varphi^{'}(r_ne^{i\theta})|r_n d\theta \\
& \leqslant diam(\partial\Omega)\|T\|\\
\end{split}
\end{equation*}
which implies that 
$$\int_0^{2\pi}|\varphi^{'}(r_ne^{i\theta})|d\theta \leqslant 2diam(\partial\Omega)\|T\|.$$
Thus, $\varphi^{'} \in H^{1}$ and $\partial\Omega$ is rectifiable.

To prove the curve $\partial\Omega$ is Ahlfors-regular, it is sufficient to show that  $\Lambda^{1}(D(w, \lambda) \cap \partial\Omega) \leqslant \|T\|\lambda$ for 
any $w \in \partial\Omega$ and $0 < \lambda < \infty$.  Since 
$$\Lambda^{1}(D(w, \lambda) \cap \Gamma_n) = \int_0^{2\pi}|\varphi^{'}(r_n e^{i\theta})|r_n \chi_{E_{r_n, \lambda}} d\theta \leqslant \|T\|\lambda,$$
where $E_{r_n, \lambda} = \{\theta \in (0, 2\pi]; \, \varphi(r_ne^{i\theta}) \in D(w, \lambda)\}$. It follows from $\varphi^{'} \in H^{1}$ that for almost all $\theta \in E_{1, \lambda}$, 
$\varphi^{'}(r_ne^{i\theta})$ has a nontangential limit $\varphi^{'}(e^{i\theta})$. Thus, we conclude by Fatou's lemma that, 
$$\Lambda^{1}(D(w, \lambda) \cap \partial\Omega) = \int_{E_{1, \lambda}}|\varphi^{'}(e^{i\theta})|d\theta \leqslant \|T\|\lambda.$$
This proves the assertion when $\partial\Omega$ is bounded.

Now suppose $\partial \Omega$ is unbounded.  Without loss of generality, we may assume $\varphi (0) = 0$. 
It is easy to see that the Ahlfors-regular constants of  $\gamma_n$ are uniformly bounded by a universal constant $C$. Thus, we conclude by Lemma 3 that the Ahlfors-regular constants of $\Gamma_n$ are uniformly bounded by the constant $C_1 \overset{{\rm def}}{=} 2max\{3C\|T\|, 27C\}$. By (\ref{def2})  the image $L_n$ of $\Gamma_n$ under the mapping  $z \mapsto 1/z$ are also Ahlfors-regular with the same uniform constant $C_1$. It follows that $L_n$ are rectifiable with uniformly bounded length.  Let $\psi (z) = 1/\varphi(\frac{1}{z})$ be a conformal homeomorphism of $\Delta^{*}$ into $\mathbb{C}$. Then we have $\psi^{'}(z) \in H^1 $ in $\Delta^{*}$ and the boundary of the image $D$ of the  domain $\Omega$ under the mapping $w \mapsto 1/w$ is rectifiable. Then the proof goes as the first case. We can obtain that $\partial D$ is an Ahlfors-regular curve. It follows from (\ref{def2}) that $\partial \Omega$ is also an Ahlfors-regular curve. This proves the assertion. 
\end{proof}

\begin{remark}
When $\partial\Omega$ is bounded, the uniform boundedness of $l_n = \Lambda^1(\Gamma_n)$ can also be proved by using an explicit construction. 

Suppose $l_n \to \infty$ as $n \to \infty$. Take $\varepsilon_n = \frac{1}{l_n} - \frac{1}{l_{n+1}},\;n = 1, 2, 3, \cdots.$ Then
$\sum\varepsilon_n = 1/l_1$ implies $\sum\varepsilon_n$ is convergent. Let $d\nu = d\sum\varepsilon_n\Lambda^1|_{\gamma_{n + 1}}$. For any $z \in \mathbb{S}$ and 
$0 < h < 1$, we have 
$$\nu(S) \leqslant 2h\sum\varepsilon_n$$
for every sector
$$S = \{re^{i\theta}:\; 1 - h \leqslant r < 1, \; |\theta - \theta_0| \leqslant h\}.$$
Thus, $\nu \in CM(\Delta)$. On the other hand, noting that $\sum\varepsilon_n l_{n+1} > \int_0^{1/l_1}\frac{1}{x}dx$, we have $\sum\varepsilon_n l_{n+1}$ is divergent. We conclude that $$d\mu = |(\varphi^{-1})^{'}|^{-1} d\nu\circ\varphi^{-1} = d\sum\varepsilon_n\Lambda^1|_{\Gamma_{n + 1}}.$$
Then $\mu$ is not a Carleson measure on $\Omega$. Contradiction.
\end{remark}

\section{On the reflections of chord-arc curves and asymptotically smooth curves}

In this section, we show that admitting a quasiconformal reflection, which is ``conformal"  near the curve in a sense given precisely by a Carleson measure condition, characterizes the chord-arc curve with small norm, and that the similar result, with the Carleson measure replaced by the vanishing Carleson measure, characterizes the asymptotically smooth curve in some condition. 

\subsection{Chord-arc curves}

A locally rectifiable curve $\Gamma$ is a chord-arc curve if $l_{\Gamma}(z_1, z_2) \leqslant k|z_1 - z_2|$ for all $z_1, z_2 \in \Gamma$, where $l_{\Gamma}(z_1, z_2)$ denotes the length of the shortest arc of $\Gamma$ joining $z_1$ and $z_2$. The smallest such $k$ is called the chord-arc constant. 

It is a well known fact that a chord-arc curve $\Gamma$ is the image of the unit circle under a bilipschitz mapping on the plane, that is, there exists a mapping $\rho: \mathbb{C} \to \mathbb{C}$ such that $\rho(\mathbb{S}) = \Gamma$ and $C^{-1}|z - w| \leqslant |\rho(z) - \rho(w)| \leqslant C|z - w|$ for all $z, w \in \mathbb{C}$. Bilipschitz mappings preserve Hausdorff dimension, and though they are a very special class of quasiconformal mappings, no characterization has been found in terms of their complex dilatation.  See \cite{AG}, \cite{HWG} for more results on this topic. 

The complex dilatations whose associated quasicircles are chord-arc curves with small constant are well understood. 
\addtocounter{theorem}{1}
\begin{theorem}[see \cite{AZ}, \cite{Se}]\label{chord}
Let $\varphi$ be a conformal mapping from the unit disk $\Delta$ onto a simply connected domain $\Omega$. Then $\Gamma = \partial\Omega$ is a chord-arc curve with small constant if and only if $\varphi$ has a quasiconformal extension with complex dilatation $\nu$ such that 
$|\nu|^2/(|z|^2 - 1)$ is a Carleson measure with small norm.
\end{theorem}

Let $\Gamma$ be a Jordan curve bounding the domains $\Omega$ and $\Omega^{*}$. A sense-reversing quasiconformal mapping $f$ of the plane which maps $\Omega$ onto $\Omega^{*}$ is a quasiconformal reflection in $\Gamma$ if $f$ keeps every point of $\Gamma$ fixed. It is well known that a Jordan curve admits a quasiconformal reflection if and only if it is a quasicircle. 

The result in \cite{Zi94} states that if a quasicircle $\Gamma$ admits a quasiconformal reflection $f$ whose complex dilatation $\mu$ satisfies 
$|\mu|^2dist(w, \Gamma)^{-1}$ is a Carleson measure with small norm, then  $\Gamma$ is a chord-arc curve with small constant. By applying Theorem \ref{chord} and the boundedness of the push-forward operator  induced by a conformal mapping from the unit disk $\Delta$ onto the chord-arc domain, the converse can be drawn.

\begin{theorem}\label{reflection1} Let  $\Gamma$ be a quasicircle. Then the following two statements are equivalent: 
\begin{description}
\item[(1)] $\Gamma$ admits a quasiconformal reflection $f$ whose complex dilatation $\mu$ satisfies that
$|\mu|^2dist(w, \Gamma)^{-1}$ is a Carleson measure with small norm.
\item[(2)] $\Gamma$ is a chord-arc curve with small constant.
\end{description}
\end{theorem}

\begin{proof}
$(1) \Rightarrow (2)$ was given in \cite{Zi94}. We will show that $(2) \Rightarrow (1)$. 

Suppose $\Gamma$ is a chord-arc curve with small constant bounding the domains $\Omega$ and $\Omega^{*}$. 
Let $\varphi$ be a conformal mapping from the unit disk $\Delta$ onto the domain $\Omega$.  By Theorem \ref{chord},  $\varphi$ has a quasiconformal extension  (still denoted by $\varphi$)  with complex dilatation $\nu$ such that 
$|\nu|^2/(|z|^2 - 1)$ is a Carleson measure with small norm (for example, Douady-Earle extension (see \cite{CZ}) or modified Beurling-Ahlfors extension constructed by Semmes in \cite{Se}). Set $f = \varphi \circ \frac{1}{\overline{\varphi^{-1}}}$. Clearly, $f$ is a quasiconformal reflection in $\Gamma$. Set 
$\mu(w) = \frac{\partial f}{\partial w}/\frac{\partial f}{\partial \bar{w}}$. Then, 
\begin{equation*}
|\overline{\mu(w)}| = \bigg|\nu\Big(\frac{1}{\overline{\varphi^{-1}(w)}}\Big)\frac{(\frac{1}{\varphi^{-1}})_w}{\overline{(\frac{1}{\varphi^{-1}})_w}}\bigg| = \bigg|\nu\Big(\frac{1}{\overline{\varphi^{-1}(w)}}\Big)\bigg|.
\end{equation*}
By Koebe distortion theorem, 
\begin{equation*}
\frac{|\mu(w)|^2}{dist(w, \Gamma)} \leqslant 4 \frac{|\nu(z)|^2}{|z|^2 - 1}\circ \overline{\Big(\frac{1}{\varphi^{-1}(w)}\Big)}\, \overline{\Big(\frac{1}{\varphi^{-1}(w)}\Big)_w}. 
\end{equation*}
The theorem follows by Theorem \ref{Z2}.
\end{proof}

\subsection{Asymptotically smooth curves}

A Jordan curve $\Gamma$ is called an asymptotically smooth curve in the sense of Pommerenke \cite[p.172]{Po92} if 
$\lim_{|z_1 - z_2| \to 0}l_{\Gamma}(z_1, z_2)/|z_1 - z_2| \to 1$ for any two points $z_1$ and $z_2$ of $\Gamma$,  where $l_{\Gamma}(z_1, z_2)$ denotes the length of the shortest arc of $\Gamma$ joining $z_1$ and $z_2$.

In \cite{Po78}, Pommerenke has obtained the following 
\begin{theorem}\label{smooth}
Let $\varphi$ be a conformal mapping from the unit disk $\Delta$ onto a simply connected domain $\Omega$. Then $\Gamma = \partial\Omega$ is an asymptotically smooth curve if and only if $\varphi$ has a quasiconformal extension with complex dilatation $\nu$ such that 
$|\nu|^2/(|z|^2 - 1)$ is a vanishing Carleson measure.
\end{theorem}

We will apply Theorem \ref{smooth} to give a characterization  for asymptotically smooth curves  in certain premise. Precisely, we will prove the following

\begin{theorem}\label{reflection2}
Let  $\Gamma$ be a chord-arc curve. Then the following two statements are equivalent:
\begin{description}
\item[(1)] $\Gamma$  admits a quasiconformal reflection $f$ whose complex dilatation $\mu$ satisfies that
$|\mu|^2dist(w, \Gamma)^{-1}$ is a vanishing Carleson measure.
\item[(2)] $\Gamma$ is asymptotically smooth.
\end{description}
\end{theorem}

\begin{proof}
By Theorem \ref{Omega} and Theorem \ref{smooth}, $(2) \Rightarrow (1)$ can be obtained similar to the proof of $(2) \Rightarrow (1)$ in Theorem 
\ref{reflection1}. We will show that $(1) \Rightarrow (2)$.

Let $h$ be a conformal mapping from the unit disk $\Delta$ onto $\Omega$ bounded by $\Gamma$. Define $\varphi$ by $\varphi(z) = h(z)$ in the closure of $\Delta$, and 
$\varphi(z) = f\circ\varphi(\frac{1}{\bar{z}})$ in $\Delta^{*}$. Then $\varphi$ is quasiconformal in the plane, and its complex dilatation $\nu$ satisfies
\begin{equation*}
|\nu(z)| = \bigg|\mu\Big(\varphi(\frac{1}{\bar{z}})\Big)\frac{(\varphi(\frac{1}{\bar{z}}))_{\bar{z}}}{(\overline{\varphi(\frac{1}{\bar{z}})})_z}\bigg|.
\end{equation*}
By Koebe distortion theorem, 
\begin{equation*}
\frac{|\nu(z)|^2}{|z|^2 - 1} 
\leqslant \frac{|\mu(w)|^2}{dist(w, \Gamma)} \circ \varphi\big(\frac{1}{\bar{z}}\big) \Big(\varphi\big(\frac{1}{\bar{z}}\big)\Big)_{\bar{z}}. 
\end{equation*}
It follows from Theorem \ref{Delta} that $|\nu|^2/(|z|^2 - 1)$ is a vanishing Carleson measure. From Theorem \ref{smooth} we obtain $\Gamma$ is an asymptotically smooth curve.
\end{proof}

We end this paper with the following
\begin{remark}
A natural question is whether Theorem \ref{reflection2} is still true if the chord-arc curve is replaced by the quasicircle. 
\end{remark}

\begin{center}
\begin{minipage}{140mm}
{\small{\bf Acknowledgement}. 
The first author would like to thank Professor Yuliang Shen for pointing out an error in the first 
draft of this paper. 
Research was supported by the National Natural Science Foundation of China (Grant Nos. 11501259, 11671175).
}

\end{minipage}
\end{center}

\end{document}